\documentclass{amsproc}

\usepackage{amsmath}
\usepackage{amssymb}
\usepackage{amsthm}
\usepackage{mathrsfs}

\newtheorem{theorem}{Theorem}[section]
\newtheorem{lemma}[theorem]{Lemma}
\newtheorem{proposition}[theorem]{Proposition}
\newtheorem{corollary}[theorem]{Corollary}

\theoremstyle{definition}

\numberwithin{equation}{section}

\begin{document}

\title[AMZ of non-Archimedean dynamics]{Artin-Mazur zeta functions of certain non-Archimedean dynamical systems}


\author{Junghun Lee}
\address{Graduate School of Mathematics, Nagoya University, 
Nagoya 464-8602, Japan}
\email{m12003v@math.nagoya-u.ac.jp}
\thanks{The author wishes to express his thanks to Professor Tomoki Kawahira and Professor Kohji Matsumoto. He is also grateful to Professor Yu Yasufuku and Professor Seidai Yasuda for their valuable comments.}



\date{October 18, 2015}

\begin{abstract}
In this paper, we will prove the rationality of the Artin-Mazur zeta functions of some non-Archimedean dynamical systems.
\end{abstract}

\maketitle


\section{Introduction}

A. Weil considered {\it local zeta functions}, which are derived from counting the number of rational points on algebraic varieties over finite fields, and conjectured rationality of the local zeta functions in \cite{We49}.
He also conjectured that the local zeta functions should satisfy a form of functional equations and an analogue of the {\it Riemann hypothesis}, which states that all zeros of the local zeta functions are placed in a restricted region such as a line or a circle.
His conjectures were solved by B. Dwork, A. Grothendieck, and P. Deligne. See 
\cite{Dw60}, \cite{Groth65}, and \cite{Del74} for more details.

In \cite{AM65}, as a generalization of the local zeta functions, M. Artin and B. Mazur introduced the Artin-Mazur zeta functions of {\it the theory of (discrete) dynamical systems}, which investigates the iterations of a given continuous map from a topological space to itself.
More precisely, for given topological space $X$ and continuous map $f : X \rightarrow X$, the {\it Artin-Mazur zeta function $\zeta_f(T)$ of $f$ on $X$} is defined by
$$
\zeta_f(T) : = \exp \left( \sum_{k=1}^{\infty}{\mathcal{N}_{k} \over k}T^{k} \right)
$$
where $\mathcal{N}_k$ is the number of the isolated fixed points of the {\it $k$-th iteration} $f^k$ of $f$.

A motivation of this paper is the result of A. Hinkkanen in \cite{Hinkk94}, which shows the rationality and an analogue of the Riemann hypothesis of the Artin-Mazur zeta functions of complex dynamical systems.
As Hinkkanen considered the Artin-Mazur zeta functions of rational maps on the Riemann sphere, it is natural to consider the Artin-Mazur zeta functions of rational maps on the projective line over an algebraically closed, complete, and non-Archimedean field.
More precisely, let $K$ be an algebraically closed field with a complete, multiplicative, non-Archimedean, and non-trivial norm $|\cdot|$.
The {\it projective line} $\mathbb{P}^1_K$ over $K$ is defined as the quotient space $(K^{2} \backslash \{ {\bf 0} \}) / \sim$ where $\sim$ is the equivalence relation defined by: $(z_0, z_1) \sim (w_0, w_1)$ if there exists a non-zero element $c$ in $K$ such that $(z_0, z_1) = (c \cdot w_0, c \cdot w_1)$. The projective line $\mathbb{P}^1_K$ is equipped with the quotient topology.
Note that we can identify $(z: 1)$ in $\mathbb{P}^1_K$ with $z$ in $K$ and $(1: 0)$ in $\mathbb{P}^1_K$ with $\infty$.
A {\it rational map} $f : \mathbb{P}^1_K \rightarrow \mathbb{P}^1_K$ over $K$ is a map given by $f(z) = {f_0(z) / f_1(z)}$ where $f_0$ and $f_1$ are polynomials over $K$ with no common factors. 
The {\it degree of $f$} is defined by $\max\{ \deg(f_0), \deg(f_1) \}$.
Now let us state the main result of this paper.

\begin{theorem}\label{main}
Let $f : \mathbb{P}^1_K \rightarrow \mathbb{P}^1_K$ be a rational map of degree $d \geq 2$ over $K$.
Suppose that the characteristic of $K$ is zero.
Then the Artin-Mazur zeta function $\zeta_f(T)$of $f$ on $\mathbb{P}^1_K$ is rational over $\mathbb{Q}$.
Moreover, all of the zeros of $\zeta_f(T)$ are on the unit circle.
\end{theorem}

In the rest of this paper, we assume that the characteristic of $K$ is zero.
We remark that if the degree of $f : \mathbb{P}^1_K \rightarrow \mathbb{P}^1_K$ is less than two, one can easily check that the Artin-Mazur zeta function $\zeta_f(T)$ is equal to either $1$, $(1 - T)^{-1}$, $(1 - T)^{-2}$, or $(1 - T)^{-2} \cdot (1 - T^q)^{2/q}$ for some $q \in \mathbb{N} \backslash \{ 1 \}$.

The idea of the proof of Theorem \ref{main} is the same with Hinkkanen's: counting the multiplicities of each fixed points and subtracting them from all fixed points.
In Section $2$, we will briefly review some of the standard facts on fixed points of rational maps.
We will prepare some key lemmas in Section $3$ and prove Theorem \ref{main} in Section $4$.

\section{Preliminaries}

In this section, we will see some basics of fixed points of rational maps over $K$.
We say that $f$ has $m$ fixed points at $\alpha \in K$ (resp. $\alpha = \infty$) if $f(z) - z$ (resp. $1/f(1/z) - z$) has $m$ zeros at $\alpha$ (resp. $0$). 
This constant $m$ is called the {\it multiplicity} of a fixed point $\alpha$ of $f$.
The {\it multiplier} of $f$ at a fixed point $\alpha$ is defined by 
\begin{align*}
\lambda(f; \alpha) =
\begin{cases}
(f)'(\alpha) \quad &(\alpha \in K) \\\
\displaystyle \lim_{z \rightarrow 0} { (f)'(1/z) \over z^2 \cdot f(1/z)^2} \quad &(\alpha = \infty).
\end{cases}
\end{align*}

The following proposition, which implies that the multiplier and multiplicity are invariant under change of coordinates, can be obtained from simple calculations.

\begin{proposition}\label{2.1}
Let $f : \mathbb{P}^1_K \rightarrow \mathbb{P}^1_K$ be a rational map of degree $\geq 2$, $\phi : \mathbb{P}^1_K \rightarrow \mathbb{P}^1_K$ a rational map of degree one, and $\alpha \in \mathbb{P}^1_K$ a fixed point. 
Then the following statements hold.
\begin{enumerate}
\item The element $\phi^{-1}(\alpha)$ is a fixed point of $\phi^{-1} \circ f \circ \phi$. \\
\item The multiplier $\lambda(f; \alpha)$ is equal to $\lambda(\phi^{-1} \circ f \circ \phi; \phi^{-1}(\alpha))$. \\
\item The multiplicity of $\alpha$ of $f(z)$ is equal to the multiplicity of $\phi^{-1}(\alpha)$ of $\phi^{-1} \circ f \circ \phi(z)$.
\end{enumerate}
\end{proposition}

As an immediate corollary, we have the following statement.

\begin{corollary}\label{2.2}
Let $f : \mathbb{P}^1_K \rightarrow \mathbb{P}^1_K$ be a rational map of degree $d \geq 2$. Then the number of fixed points of $f$ in $\mathbb{P}^1_K$ is exactly $d +1$, counted with multiplicity.
\end{corollary}

\begin{proof}[Proof of Corollary \ref{2.2}]
Since $f$ is not the identity map on $\mathbb{P}^1_K$, we may assume that $\infty$ is not a fixed point of $f$ by change of coordinates.
Let us write $f(z) = f_0(z) / f_1(z)$ where $f_0 , f_1$ are polynomial maps with no common factors.
One can easily check that $\alpha \in K$ is a fixed point of $f$ if and only if
$$
f_0(\alpha) - \alpha \cdot f_1(\alpha) = 0.
$$
Since $K$ is algebraically closed, $f_0 (z) - z \cdot f_1(z)$ is a polynomial over $K$, and $d = \deg(f_1)$, there exist exactly $d + 1$ zeros of $f_1(z) - z \cdot f_2(z)$ in $K$, counted with multiplicity.
By Proposition \ref{2.1}, the multiplicities of fixed points are invariant under change of coordinates thus we complete our proof.
\end{proof}


%
\begin{proposition}\label{2.3}
Let $f : \mathbb{P}^1_K \rightarrow \mathbb{P}^1_K$ be a rational map of degree $\geq 2$.
Then the following statements hold.
\begin{enumerate}
\item If $\alpha, f(\alpha) \in K$, then there exist an $r > 0$ and a sequence $\{ a_k \}_{k = 2}^{\infty} \subset K$ such that $\lim_{k \rightarrow \infty} |a_k| \cdot r^k = 0$ and
$$
f(z) = f(\alpha) + \lambda(f; \alpha) \cdot (z- \alpha) + a_2 \cdot (z - \alpha)^2 + \cdots
$$
on $\{ z \in K \mid |z - \alpha|  \leq r \}$. 
Moreover, if $\alpha = f(\alpha)$, then the multiplicity of $\alpha$ of $f$ is equal to
$$
\begin{cases}
1 \quad &(\lambda(f; \alpha) \neq 1) \\
\mu(f; \alpha) \quad &(\lambda(f; \alpha) = 1)
\end{cases}
$$
where $\mu(f; \alpha)$ is the natural number $\min\{i \in \mathbb{N} \mid i \geq 2, a_i \neq 0 \}$.

\item If $\alpha = f(\alpha) = \infty$, then there exist an $r > 0$ and a sequence $\{ a_k \}_{k = 2}^{\infty} \subset K$ such that $\lim_{k \rightarrow \infty} |a_k| \cdot r^k = 0$ and 
$$
{ 1 \over f(1/z)} = \lambda(f; \alpha) \cdot z + a_2 \cdot z^2 + \cdots
$$
on $\{ z \in K \mid |z |  \leq r \}$. 
Moreover, the multiplicity of $\alpha$ of $f$ is equal to
$$
\begin{cases}
1 \quad &(\lambda(f; \alpha) \neq 1) \\
\mu(f; \alpha) \quad &(\lambda(f; \alpha) = 1)
\end{cases}
$$
where $\mu(f; \alpha)$ is the natural number $\min\{i \in \mathbb{N} \mid i \geq 2, a_i \neq 0 \}$.
\end{enumerate}
\end{proposition}
See \cite[Proposition $5.8$(b), (c)]{Silv07} and \cite[Proposition A.16]{BR10} for the proof of Proposition \ref{2.3}.
%




\section{Key lemmas}

In this section, we prepare some lemmas for the proof of Theorem \ref{main}.
The main reason why we assume that the characteristic of $K$ is zero in Theorem \ref{main} is to obtain the finiteness theorem as follows.

\begin{lemma}\label{3.1}
Let $f : \mathbb{P}^1_K \rightarrow \mathbb{P}^1_K$ be a rational map of degree $d \geq 2$.
Then the cardinality of the set
$$
\mathcal{P} := \{ \alpha \in \mathbb{P}^1_K \mid \exists n \in \mathbb{N}, \exists q \in \mathbb{N}, f^n(\alpha) = \alpha, \lambda(f^n; \alpha)^q = 1  \} $$
is finite.
\end{lemma}
\begin{proof}[Proof of Lemma \ref{3.1}]
We write the rational map $f$ over $K$ by
$$
{ a_0 + a_1 \cdot z + \cdots + a_d \cdot z^d \over b_0 + b_1 \cdot z + \cdots + b_d \cdot z^d}
$$
and consider the field $\mathbb{Q}(\mathcal{C})$ over $\mathbb{Q}$ where $\mathcal{C} := \{ a_0, a_1, \cdots, a_d, b_0, \cdots, b_d \} \cup (\mathcal{P} \backslash \{ \infty \} )$.
Since the characteristic of $K$ is zero and $\mathcal{C}$ is countable, there exists an algebraic embedding $\iota : \mathbb{Q}(\mathcal{C}) \rightarrow \mathbb{C}$. 
Now we consider the induced rational map $\hat{f} : \mathbb{P}^1_{\mathbb{C}} \rightarrow \mathbb{P}^1_{\mathbb{C}}$ given by
$$
{ \iota(a_0) + \iota(a_1) \cdot z + \cdots + \iota(a_d) \cdot z^d \over \iota(b_0) + \iota(b_1) \cdot z + \cdots + \iota(b_d) \cdot z^d}.
$$
It follows from simple calculations that for any element $\beta \in \mathcal{P} \backslash \{ \infty \}$, the element $\iota(\beta)$ is contained in the set
$$
\hat{\mathcal{P}} := \{ \alpha \in \mathbb{P}^1_\mathbb{C} \mid \exists n \in \mathbb{N}, \exists q \in \mathbb{N}, \hat{f}^n(\alpha) = \alpha, \lambda(\hat{f}^n; \alpha)^q = 1  \}.
$$
Then it follows from \cite[Corollary $10. 16$]{Miln06} that the set $\hat{\mathcal{P}}$ is finite thus the set $\mathcal{P}$ is also finite.
\end{proof}

Now let us consider periodic points.
We say a point $\alpha \in \mathbb{P}^1_K$ is a {\it periodic point} of $f$ with the minimal period $n \in \mathbb{N}$ if $f^n(\alpha) = \alpha$ and $f^l(\alpha) \neq \alpha$ for any $l \in \{1, 2, \cdots, n - 1 \}$.

\begin{lemma}\label{3.2}
Let $f : \mathbb{P}^1_K \rightarrow \mathbb{P}^1_K$ be a rational map of degree $\geq 2$ and $\alpha$ be a periodic point of $f$ with the minimal period $n$.
Suppose that $\{ f^{k}(\alpha) \}_{k = 0}^{n - 1}$ is contained in $K$.
Then the multiplicity of $\alpha$ of $f^n$ is equal to the multiplicity of $f^k(\alpha)$ of $f^n$ for any $k \in \{ 0, 1, \cdots, n-1 \}$.
\end{lemma}
\begin{proof}[Proof of Lemma \ref{3.2}]
It follows from Proposition \ref{2.3}(1) that $f$ can be written by
\begin{align*}
f(\alpha_k) + \lambda_k \cdot (z - \alpha_k) + a_{\mu_k} \cdot {(z - \alpha_k)^{\mu_k}} + \cdots
\end{align*}
near each $\alpha_k$ thus $f^n$ can be written by 
$$
\alpha_k +  \left( \prod_{l =1}^{n} \lambda_l  \right) \cdot (z - \alpha_k) + \sum_{l = 0}^{n - 1} \Lambda(\mu_l) + \cdots
$$
near each $\alpha_k$ where
$$
\Lambda(\mu_l) = {\lambda_1}^{\mu_l} \cdot {\lambda_2}^{\mu_l} \cdot \cdots \cdot  {\lambda_{l-1}^{\mu_l}} \cdot a_{\mu_l} \cdot {\lambda_{l + 1}} \cdot \cdots \cdot \lambda_n \cdot (z - \alpha_l)^{\mu_l}.
$$
Since $\prod_{l =1}^{n} \lambda_l $ and $\sum_{l = 0}^{n - 1} \Lambda(\mu_l)$ are independent of $k$, we complete our proof.
\end{proof}
\begin{lemma}\label{3.3}
Let $f : \mathbb{P}^1_K \rightarrow \mathbb{P}^1_K$ be a rational map of degree $\geq 2$ and $\alpha \in \mathbb{P}^1_K$ be a fixed point of $f$.
Suppose that there exists a natural number $q$ such that $\lambda(f; \alpha)^q = 1$ and $\lambda(f; \alpha)^l \neq 1$ for any $l$ in $\{ 1, 2, \cdots, q-1 \}$. 
Then $\mu(f^q; \alpha) - 1$ can be divided by $q$.
Moreover, for any natural number $k$, we have $\mu(f^q; \alpha) = \mu(f^{k \cdot q}; \alpha)$.
\end{lemma}
\begin{proof}[Proof of Lemma \ref{3.3}]
Without loss of generality, we may assume that $\alpha = f(\alpha) = 0$. 
Then $f$ can be written by 
$$
\lambda \cdot z + a_{\mu} \cdot z^{\mu} + \cdots
$$
near $\alpha$ where $\lambda : = \lambda(f; \alpha)$ and $\mu := \mu(f; \alpha)$ and $f^q$ can be written by
$$
\lambda^q \cdot z + b_{\tilde{\mu}} \cdot z^{\tilde{\mu}} + \cdots
$$
near $\alpha$ where $\tilde{\mu} := \mu(f^q; \alpha)$.
We first calculate that
\begin{align*}
f^{q} \circ f (z) 
&= \lambda^q \cdot f(z) + b_{\tilde{\mu}} \cdot f(z)^{\tilde{\mu}} + \cdots \\
&= \lambda^q \cdot (\lambda \cdot z + a_\mu \cdot z^\mu + \cdots + a_{\tilde{\mu}}\cdot z^{\tilde{\mu}} + \cdots) \\
&\quad + b_{\tilde{\mu}} \cdot (\lambda \cdot z + a_\mu \cdot z^{\mu } + \cdots + a_{\tilde{\mu}} \cdot z^{\tilde{\mu}} + \cdots)^{\tilde{\mu}} + \cdots \\
&= \lambda^{q+1} \cdot z + a_\mu \cdot \lambda^q \cdot z^{\mu} + \cdots + ( a_{\tilde{\mu}} \cdot \lambda^q + b_{\tilde{\mu}} \cdot \lambda^{\tilde{\mu}}) \cdot z^{\tilde{\mu}} + \cdots.
\end{align*}
Next we compute that
\begin{align*}
f \circ f^{q} (z) 
&= \lambda \cdot f^{q}(z) + a_\mu \cdot f^q(z)^\mu + \cdots + a_{\tilde{\mu}} \cdot f^{q}(z)^{\tilde{\mu}} + \cdots \\
&= \lambda \cdot (\lambda^q \cdot z + b_{\tilde{\mu}} \cdot z^{\tilde{\mu}} + \cdots ) + a_{\mu} \cdot (\lambda^q \cdot z + b_{\tilde{\mu}} \cdot z^{\tilde{\mu}} + \cdots )^{\mu} \\
&\quad + \cdots +  a_{\tilde{\mu}} \cdot (\lambda^q \cdot z + b_{\tilde{\mu}} \cdot z^{\tilde{\mu}} + \cdots )^{\tilde{\mu}} + \cdots\\
&= \lambda^{q + 1} \cdot z + a_{\mu} \cdot \lambda^{q \mu} \cdot z^{\mu} + \cdots + (b_{\tilde{\mu}} \cdot \lambda +  a_{\tilde{\mu}} \cdot \lambda^{q \tilde{\mu}} ) \cdot z^{\tilde{\mu}} + \cdots.
\end{align*}
It follows from a trivial functional equation $f^{q} \circ f = f^{q + 1} = f \circ f^{q}$ that
$$
a_{\tilde{\mu}} + b_{\tilde{\mu}} \cdot \lambda^{\tilde{\mu}} = b_{\tilde{\mu}} \cdot \lambda + a_{\tilde{\mu}} \cdot \lambda^{q \tilde{\mu}} = b_{\tilde{\mu}} \cdot \lambda + a_{\tilde{\mu}}.
$$
This implies that $\lambda^{\tilde{\mu} - 1} = 1$ thus $\tilde{\mu} - 1$ can be divided by $q$.

To prove the second statement, let us see the following equations.
\begin{align*}
f^{k q}(z) = \underbrace{f^q \circ f^q \circ \cdots \circ f^q}_{k \text{ times} }(z)
&= z + (\underbrace{b_{\tilde{\mu}} + \cdots + b_{\tilde{\mu}}}_{k \text{ times}})  \cdot z^{\tilde{\mu}} + \cdots.
\end{align*}
It follows from Proposition \ref{2.3} that $\mu(f^{k q}; \alpha ) = \tilde{\mu} = \mu(f^q; \alpha)$ for any $k \in \mathbb{N}$.
\end{proof}
%

\section{Proof of Theorem \ref{main}}

We end this paper with the proof of the main theorem.

\begin{proof}[Proof of Theorem \ref{main}]
Without loss of generality, we may assume that $\infty$ is a fixed point of $f$.
By Corollary \ref{2.2}, the number of fixed points of $f^n$ is $d^n + 1$, counted with multiplicity. Let us count the multiplicity of a fixed point $\alpha$ of $f^n$.

If $\alpha \notin \mathcal{P}$ where $\mathcal{P}$ is the set defined in Lemma \ref{3.1}, then $\lambda(f^{k n}; \alpha) = \lambda(f^{n}; \alpha)^k \neq 1$ for any $k \in \mathbb{N}$ thus it follows from Proposition \ref{2.3} that the multiplicity of $\alpha$ of $f^{k n}$ is equal to $1$.
On the other hand, by Lemma \ref{3.1}, the cardinality of $\mathcal{P}$ is finite thus we can decompose $\mathcal{P}$ into $N$ {\it finite invariant sets} $\{ C_i \}_{i = 1}^{N}$, that is, each $C_i$ is finite, $f(C_i) = C_i$, and
$$
\mathcal{P} = C_1 \sqcup C_2 \sqcup \cdots \sqcup C_N.
$$
Let us write each set as $C_i = \{ f^{l}(\alpha_i) \}_{l = 0}^{n_i - 1}$ where $n_i$ is the minimal period of the periodic point $\alpha_i$.
Let $q_i$ be the minimal natural number satisfying $\lambda(f^{n_i}; \alpha_i)^{q_i} = 1$ and $r_i$ the natural number satisfying $q_i \cdot r_i = \mu(f^{n_i q_i}; \alpha_i) - 1$.
For each $i \in \{1, 2, \cdots, N \}$, the multiplicity of $\alpha_i$ of $f^{m n_i}$ is equal to $1$ (resp. $\mu(f^{m n_i}; \alpha_i) = \mu(f^{k n_i q_i}; \alpha_i) = \mu(f^{n_i q_i }; \alpha_i) = 1 + q_i \cdot r_i$) when $m \neq k \cdot q_i$ for any $k \in \mathbb{N}$ (resp. $m = k \cdot q_i$ for some $k \in \mathbb{N}$) by Proposition \ref{2.3} and Lemma \ref{3.3}.
Furthermore, if $n_i \geq 2$, then $C_i \subset K$ because $\infty$ is a fixed point of $f$. 
Thus it follows from Lemma \ref{3.2} that the multiplicity of $\alpha_i$ of $f^{n_i}$ is equal to the multiplicity of $f^l(\alpha_i)$ of $f^{n_i}$ for any $l \in \{0, 1, \cdots, n_i -1 \}$.

Finally, we obtain the closed form of $\zeta_f$ by subtracting the multiplicity from the total number of periodic points as follows.
\begin{align*}
\sum_{j=1}^{\infty}{\mathcal{N}_j \over j}T^j 
&= \sum_{j=1}^{\infty}{d^j + 1 \over j}T^j - \sum_{i = 1}^{N} \left( \sum_{k_i = 1}^{\infty} { n_i \cdot q_i \cdot r_i \over n_i \cdot q_i \cdot k_i} T^{n_i q_i k_i} \right)  \\
&=  \sum_{j=1}^{\infty}{d^j \over j}T^j + \sum_{j =1}^{\infty}{1 \over j}T^j 
- \sum_{i = 1}^{N} r_i \left( \sum_{k_i =1}^{\infty} {  1 \over k_i} (T^{n_i q_i})^{k_i} \right) \\
&= \log(1 - dT)^{-1} + \log(1 - T)^{-1} + \sum_{i = 1}^{N} r_i \cdot \log(1 - T^{n_i q_i}) \\
&= \log \{ (1 - dT)^{-1} \cdot (1 - T)^{-1} \cdot \prod_{i = 1}^{N}(1 - T^{n_i q_i})^{r_i} \}.
\end{align*}
This implies that
$$
\zeta_{f}(T) = \exp \left( \sum_{j=1}^{\infty}{\mathcal{N}_j \over j}T^j \right) = (1 -dT)^{-1} \cdot (1 - T)^{-1} \cdot \prod_{i = 1}^{N} (1 - T^{n_i q_i})^{r_i}.
$$
In particular, this implies that $\zeta_f(T)$ is rational over $\mathbb{Q}$ and all of the zeros of $\zeta_f(T)$ are on the unit disk.
\end{proof}


\bibliographystyle{amsalpha}


\end{document}